\documentclass[11pt]{article}
\usepackage{setspace}
\usepackage{mathrsfs}
\usepackage{mathrsfs}
\usepackage{mathrsfs}
\usepackage{graphicx}
\usepackage{epstopdf}
\usepackage{multirow}
\usepackage{cite}
\usepackage{enumitem}
\usepackage{authblk}
\usepackage{amsfonts}
\usepackage{amssymb}
\usepackage{amsmath}
\usepackage{amsthm}
\usepackage[notref,notcite]{showkeys}
\usepackage[colorlinks=true]{hyperref}
\usepackage{hyperref}
\hypersetup{linkcolor=blue}
\allowdisplaybreaks[4]
\theoremstyle{plain}

\numberwithin{equation}{section}
\newtheorem{theorem}{Theorem}[section]
\newtheorem{lemma}[theorem]{Lemma}
\newtheorem{proposition}[theorem]{Proposition}

\theoremstyle{definition}
\newtheorem{defn}{Definition}[section]

\theoremstyle{remark}
\newtheorem{remk}{Remark}[section]

\begin{document}

\title{{\Large\bf{Normalized solutions of quasilinear Schrödinger equation with Sobolev critical exponent on star-shaped bounded domains}}}

\author{Ru Yan$^{\mathrm{a,b}}$ \\
{\small $^{\mathrm{a}}$Institute of Mathematics, AMSS, Chinese Academy of Science,
Beijing 100190, China}\\
{\small $^{\mathrm{b}}$University of Chinese Academy of Science,
Beijing 100049, China}\\
}
\date{}
\maketitle

\begin{minipage}{14cm} {\bf Abstract:}
	In this paper, we consider a quasilinear Schrödinger equation with critical exponent on bounded domains. Via a dual approach, we establish the existence of two positive normalized solutions: one is a ground state and the other is a mountain pass solution.     
\end{minipage}

\section{Introduction}
In this paper, we consider the following quasilinear Schrödinger equation:
\begin{equation}\label{5aa}
    \begin{cases}
        -\Delta u - u\Delta (u^2) - |u|^{2(2^*)-2}u = \lambda u, & x \in \Omega, \\
        u > 0 \ \text{in}\ \Omega, \ u = 0 \ \text{on}\ \partial \Omega, \ \int_{\Omega} |u|^2 \, \mathrm{d}x = c.
    \end{cases}
\end{equation}
where $c>0,\ \Omega\subset \mathbb{R}^N(N\ge 3)$ is smooth, bounded, star-shaped and $2^*=\frac{2N}{N-2}$. $2\cdot 2^*$ was known to be a critical exponent for equation \eqref{5aa}, which is first observed in \cite{wzq2004}. The solutions of \eqref{5aa} give rise to standing waves of the time-dependent quasilinear Schrödinger equation for the following:
\begin{equation}
    \begin{cases}
        i\partial_t \varphi + \Delta \varphi + \varphi \Delta (|\varphi|^2) + |\varphi|^{p-1} \varphi = 0, & \text{in } \mathbb{R} \times \Omega,\\
        \varphi(0,x) = \varphi_0(x), & \text{in } \Omega.
    \end{cases}
\end{equation}
Equations of this type arise in various branches of physics, such as plasma physics, and fluid mechanics. For further details on the relevant physical background, we refer the reader to references \cite{jean2010,wzq2002} and the literature cited therein. 

When $\lambda$ appears as a fixed parameter, the solutions of \eqref{5aa} correspond to the critical points of the energy functional given by 
$$
E_\lambda (u)=\frac{1}{2}\int_{\Omega}|\nabla u|^2 d x+\frac{\lambda}{2}\int_{\Omega}|u|^2 d x+\int_{\Omega}u^2|\nabla u|^2 dx-\frac{1}{2\cdot2^*}\int_{\Omega}|u|^{2(2^*)}dx
$$
which is defined on the natural space 
$$
\mathcal{X}:= \left\{ u \in H_0^1(\Omega) : \int_{\Omega} u^2|\nabla u|^2 \, dx < \infty \right\}.
$$
Since the term $u\Delta (u^2)$ causes $E_\lambda$ to be non-differentiable in $\mathcal{X}$, the search for solutions of \eqref{5aa} is considerably more difficult than in the semilinear case. In order to overcome this difficulty, several variational techniques have been developed, such as minimization methods \cite{wzq2003,wzq2002,Ruiz2010}, the dual approach \cite{jean2004,wzq2003-2}, the Nehari manifold approach \cite{wzq2004},  and
 perturbation approaches \cite{wzq2014,wzq2013}.

 Due to the fact that normalized solutions carry stronger physical significance, considerable research has been conducted on normalized solutions of quasilinear equations. We refer the reader to \cite{wzq2023, Zou2023, Yang2025, zxx2024, chang2025, Gyx2024, Gyx2025,jean2015} and the references therein for more information. To the best of our knowledge, no one has considered normalized solutions on bounded domains for the quasilinear Schrödinger equation. Compared with the entire space, tools such as dilations \cite{jean1997} and the Phonzaev mafold \cite{Bartch2017} are no longer applicable on bounded domains. In recent years, there have been new advances to normalized solutions of the semilinear Schrödinger equation\cite{Pierotti2017, Noris2019, Noris2014, chang20252, Pellacci2021, Bartsch2024, Lyy2025, song2024, song2023}. Inspired by \cite{song2024}, we investigate normalized solutions of \eqref{5aa}. 
 
 Firstly, we prove the existence of a normalized ground state in the following sense.
 \begin{defn}
    Let 
    \begin{align*}
    \tilde{S}_c:=\{u \in \mathcal{X} \text { s.t. } \int_{\Omega}\left|u\right|^2 d x=c\}.
    \end{align*}
    We say that $u\in \tilde{S}_c$ is a normalized ground state of \eqref{5aa}, if it is a solution that has minimal energy among all solutions belonging to $\tilde{S}_c$. Namely, if 
    $$
    \tilde{E} (u_c)=\inf \left \{ \tilde{E} (u):u\in \tilde{S}_c,\left(\tilde{E} |_{S_c}\right)^{\prime}\left(u\right)=0\right \}, 
    $$
    where
    $$
    \tilde{E}(u)=\frac{1}{2}\int_{\Omega}|\nabla u|^2 d x+\int_{\Omega}u^2|\nabla u|^2 dx-\frac{1}{2\cdot 2^*}\int_{\Omega}|u|^{2(2^*)}dx.
    $$
\end{defn}

\begin{theorem}\label{A}
    Let $\Omega$ be smooth, bounded, and star-shaped with respect to the origin, let $\mathcal{S}$ be the Sobolev best constant, and let $\lambda_1(\Omega)$ denote the first eigenvalue of $-\Delta$ on $\Omega$ with Dirichlet boundary condition. Then there exists $c_0>0$ such that for $0<c<c_0$, \eqref{5aa} has a normalized ground state $u_c$ with Lagrange multiplier $\lambda_c$ satisfying $\lambda_c < \lambda_1(\Omega)$. Moreover, $u_c$ at the level $\nu_c \in \left ( 0, \frac{1}{N}\left ( \frac{\mathcal{S}}{2} \right ) ^\frac{N}{2} \right )$.  
\end{theorem}

\begin{remk}
    Due to the presence of quasilinear term, the natural energy functional is non-differentiable and variational method cannot be applied directly. To overecome this difficulty, we apply the dual approach to reformulate the original problem into a semilinear one. Based on the favorable properties of the variable substitution and combined with the Pohozaev identity, we define a new set $\mathcal{G}$. We can find a minimizer on $\mathcal{G}$ and establish the existence of the positive normalized ground state for \eqref{5aa}.    
\end{remk}

\begin{theorem}\label{B}
    According to the assumptions of Theorem \ref{A}, \eqref{5aa} has the second positive solution $\tilde{u}_c\ne u_c$.
\end{theorem}

\begin{remk}
    Note $u_c$ is a local minimizer on $\tilde{S}_c$. Based on this, we can further construct a mountain pass structure (Lemma \ref{G}). To obtain the normalized mountain pass solution, we need to find a bounded Palais-Smale sequence and recover compactness, each of which seems hard to verify. We apply Jeanjean’s monotonicity trick to find a bounded Palais-Smale sequence. Additionally, the energy estimate of the mountain pass solution plays a crucial role in recovering compactness.
\end{remk}

The paper is organized as follows. In Section 2, we reformulate the quasilinear problem as a semilinear one using a change of variables. In Section 3, we prove the existence of a normalized ground state (Theorem \ref{A}). In Section 4, we prove the existence of a normalized mountain pass solution (Theorem \ref{B}).

\section{Preliminarise}
To find solutions of \eqref{5aa}, we search for critical points of the energy functional $\tilde{E}(u)$ under the constraint $\tilde{S}_c$. However, 
$\tilde{E}(u)$ is not well defined in $H_0^1(\Omega)$. In order to apply variational methods, inspired by \cite{jean2004, wzq2003-2}, we make a change of variables by setting $v=f^{-1}(u)$, where $f$ is defined by 
\begin{equation}
    \left\{\begin{aligned}
        &f(0)=0,\\
        &f'(t)=(1+2|f(t)|^2)^{-\frac{1}{2} },\ t>0,\\
        &f(t)=-f(-t),\ t<0.
    \end{aligned}
    \right.\nonumber
\end{equation}
Then $f$ is uniquely defined, smooth, and invertible. Furthermore, $f$ satisfies the following properties, which have been proved in \cite{jean2004, wzq2003-2,Do2009}.
\begin{lemma}\label{C}
    (1) $|f'(t)| \leq 1$ for all $t \in \mathbb{R}$;\\

    (2) $|f(t)| \leq |t|$ for all $t \in \mathbb{R}$;\\

    (3) $|f(t)| \leq 2^{\frac{1}{4}} |t|^{\frac{1}{2}}$ for all $t \in \mathbb{R}$;\\

    (4) $\frac{1}{2} f(t) \leq tf'(t) \leq f(t)$ for all $t \in \mathbb{R}^+$;\\

    (5) $\frac{1}{2} f^2(t) \leq f(t) f'(t) t \leq f^2(t)$ for all $t \in \mathbb{R}$;\\

    (6) there exists a positive constant $C$ such that
    $$
    |f(t)| \geq 
    \begin{cases}
    C|t|, & \text{if } |t| \leq 1, \\
    C|t|^{\frac{1}{2}}, & \text{if } |t| \geq 1.
    \end{cases}
    $$
\end{lemma}
Define 
\begin{align*}
    \tilde{I}(v):= \tilde{E}(f(v))=\frac{1}{2}\int_{\Omega}|\nabla v|^2 d x-\frac{1}{2\cdot2^*}\int_{\Omega}|f(v)|^{2(2^*)}dx.
\end{align*}
The functional $\tilde{I}(v)$ is well defined in $H_0^1(\Omega)$ and belongs to $C^1$. It is well known that the critical points of $\tilde{I}(v)$ are weak solutions of the semilinear elliptic equation:
\begin{equation}\label{5ad}
    \left\{\begin{aligned}
        &-\Delta v-|f(v)|^{2(2^*)-2}f(v)f'(v) =\lambda f(v)f'(v),\ x\in\Omega,\\
        &v>0\ \text {in}\ \Omega,\ v=0\ \text {on}\ \partial \Omega,\  \int _{\Omega}|f(v)|^2dx=c.\\
    \end{aligned}
    \right.
\end{equation}
Moreover, if $v$ is a critical point of $\tilde{I}(v)$, then $u=f(v)$ is a weak solution of \eqref{5aa}.

To find positive solutions of \eqref{5ad}, we search for critical points of the energy  
\begin{align*}
    I(v)=\frac{1}{2}\int_{\Omega}|\nabla v|^2 d x-\frac{1}{2\cdot2^*}\int_{\Omega}|f(v^+)|^{2(2^*)}dx,
\end{align*}
under the constraint $S_c^+$, where 
$$
S_c^+:=\{v \in H_0^1(\Omega)\ \text {s.t.} \int_{\Omega}|f(v^+)|^2 d x=c\}.
$$
Note 
\begin{align*}
    &E(u)=\frac{1}{2}\int_{\Omega}|\nabla u|^2 d x+\int_{\Omega}u^2|\nabla u|^2 dx-\frac{1}{2\cdot 2^*}\int_{\Omega}|u^+|^{2(2^*)}dx,\\
    &\tilde{S}_c^+:=\{u \in \mathcal{X}\ \text {s.t.} \int_{\Omega}|u^+|^2 d x=c\}.
\end{align*}

Now, we introduce a key set $\mathcal{G}_c$. As we all know, any critical point $v$ of $\left.I\right|_{S_c^+}$ satisfies the following Pohozaev identity:
\begin{align*}
    \frac{N-2}{2} \int_{\Omega}\left|\nabla v\right|^2 dx+\frac{1}{2}\int_{\partial \Omega} \left|\nabla v\right|^2\sigma\cdot n d\sigma =\frac{N\lambda}{2}\int_{\Omega}|f(v^+)|^2dx + \frac{N}{2\cdot2^*}\int_{\Omega}|f(v^+)|^{2(2^*)}dx.
\end{align*}
Furthermore, $v$ is a solution of \eqref{5ad}, we also have
\begin{align*}
    \int_{\Omega}\left|\nabla v\right|^2 dx =\lambda\int_{\Omega}f(v^+)f'(v^+)v^+dx + \int_{\Omega}|f(v^+)|^{2(2^*)-2}f(v^+)f'(v^+)v^+dx.
\end{align*}
Note that $\sigma\cdot n > 0$ since $\Omega$ is star-shaped with respect to the origin. By Lemma \ref{C}, we get $v \in \mathcal{G}$, where 
$$
\mathcal{G}:=\left\{v \in S_c^+:  \int_{\Omega}\left|\nabla v\right|^2 dx>\frac{1}{2}\int_{\Omega}|f(v^+)|^{2(2^*)}dx\right\}.\\
$$
For any $v \in \mathcal{G}$, 
\begin{align}
    I(v)>\left (\frac{1}{2}-\frac{1}{2^*} \right )\int_{\Omega}\left|\nabla v\right|^2 dx>\frac{1}{N}\int_{\Omega}\left|\nabla v\right|^2 dx.\label{5ac}  
\end{align}
which implies that $I$ is bounded from below on $\mathcal{G}$. Further, note that 
$$
\partial \mathcal{G}:=\left\{v \in S_c^+: \int_{\Omega}\left|\nabla v\right|^2 dx =\frac{1}{2}\int_{\Omega}|f(v^+)|^{2(2^*)}dx\right\}.
$$
For any $v \in \partial \mathcal{G}$, by Lemma \ref{C}, we have 
\begin{align}
    \int_{\Omega}\left|\nabla v\right|^2 dx =\frac{1}{2}\int_{\Omega}|f(v^+)|^{2(2^*)}dx\le 2^{\frac{2^*}{2}}\int_{\Omega}|v^+|^{2^*}dx\le \left ( \frac{\mathcal{S}}{2}  \right ) ^{-\frac{2^*}{2}}(\int_{\Omega}\left|\nabla v\right|^2 dx)^\frac{2^*}{2}. \label{5ae}
\end{align}
Hence, $\int_{\Omega}\left|\nabla v\right|^2 d x$ has a positive lower bound on $\partial \mathcal{G}$.






\section{Positive normalized ground state solution}
Based on the observations about $\mathcal{G}$ in Section 2, we consider the following constrained minimization problem:
$$
\nu_c:=\underset{v\in\mathcal{G}}{\inf} I(v).
$$
Before proving Theorem \ref{A}, we introduce an important result.
\begin{lemma}\label{D}
    Let $\Omega$ be smooth, bounded, and star-shaped with respect to the origin.
    If $v$ is a critical point of $\left.I\right|_{S_c^+}$, then $v \in \mathcal{G}$. Furthermore, we have\\
	$\mathrm{(1)}$ Any sequence $\{v_n\}\subset \mathcal{G}$ satisfying $\limsup _{n \rightarrow \infty} I\left(v_n\right)<\infty$ is bounded in $H_0^1(\Omega).$\\
    $\mathrm{(2)}$ There exists $c_0$ small, let $0<c<c_0$. Then $\mathcal{G} \neq \emptyset$, and it holds
    \begin{align}
        0<\inf _{v \in \mathcal{G}} I(v)<\inf _{v \in \partial \mathcal{G}} I(v).\label{5af}
    \end{align}
\end{lemma}
\begin{proof}
    \textbf{$(1)$} Let $(v_n) \subset \mathcal{G}$ satisfy $\underset{n \rightarrow \infty}{\limsup} I\left(v_n\right)<\infty $. Using \eqref{5ac}, we get
    $$
    \limsup _{n \rightarrow \infty} \int_{\Omega}|\nabla v_n|^2 dx \le N \limsup _{n \rightarrow \infty}I\left(v_n\right)<\infty,
    $$
    implying that ${v_n}$ is bounded in $H_0^1(\Omega)$.\\
    \textbf{$(2)$} Take $c_1$ small such that for any $u \in \tilde{S}_1^+$, $0<c<c_1$,  
    \begin{align*}
        &\int_{\Omega}|\nabla(\sqrt{c} u)|^2 d x +2\int_{\Omega}(\sqrt{c} u)^2|\nabla(\sqrt{c} u)|^2 d x\\
        =&c \int_{\Omega}|\nabla u|^2 dx +2c^2\int_{\Omega}u^2|\nabla u|^2 d x>\frac{c^{2^*}}{2}\int_{\Omega}|u|^{2(2^*)}dx=\int_{\Omega}|\sqrt{c}u|^{2(2^*)}dx. 
    \end{align*}
    Let $v=f^{-1}(\sqrt{c}u)$. Then $v \in \mathcal{G}$ yielding that $\mathcal{G}$ is not empty. Further, using \eqref{5ae}, one gets 
    $$
    \inf _{v \in \partial \mathcal{G}} \int_{\Omega}|\nabla v|^2 dx \geq \left ( \frac{\mathcal{S}}{2} \right ) ^\frac{N}{2}.
    $$
    Then by \eqref{5ac} we have 
    $$
    \inf _{v \in \partial \mathcal{G}}I(v)\ge \frac{1}{N}\int_{\Omega}\left|\nabla v\right|^2 dx\ge \frac{1}{N}\left ( \frac{\mathcal{S}}{2} \right ) ^\frac{N}{2}.
    $$
    For any $u\in \tilde{S}_1^+$, let $v=f^{-1}(\sqrt{c} u)$. Take $c_2$ small such that for $0<c<c_2$, we have 
    \begin{align*}
        I(v)&=E (\sqrt{c} u)<\frac{1}{2}\left ( \int_{\Omega}c|\nabla u|^2 d x + 2c^2\int_{\Omega}u^2|\nabla u|^2 d x\right )\\  
        &<\frac{1}{N}\left (\frac{\mathcal{S}}{2}\right)^\frac{N}{2} < \inf _{v \in \partial \mathcal{G}} I(v).
    \end{align*}
    When $c<\min (c_1, c_2)$, we can take some $u\in \tilde{S}_1^+$ such that $v_1=f^{-1}(\sqrt{c} u)\in \mathcal{G}$ and 
    $$
    \inf _{v \in \mathcal{G}}I(v)\le I(v_1)<\inf _{v \in \partial \mathcal{G}}I(v).
    $$
    Using \eqref{5ac} again, we have   
    $$
    \inf _{v \in \mathcal{G}}I(v)\ge \frac{1}{N}\int_{\Omega}|\nabla v|^2 dx \ge \frac{1}{N}\lambda_1c>0.
    $$
    Hence, \eqref{5af} holds true.      
\end{proof}

\begin{proof}[proof of Theorem \ref{A}]
    The value range of $\nu_c$ was established in the proof of Lemma \ref{D}. We now demonstrate that $\nu_c$ is attainable within $\mathcal{G}$. Consider a minimizing sequence $\{v_n\} \subset \mathcal{G}$ such that $I(v_n) \rightarrow \nu_c$ as $n \rightarrow \infty$. Lemma \ref{D}(1) implies that $\{v_n\}$ is bounded in $H_0^1(\Omega)$. Lemma \ref{D}(2) guarantees that $\{v_n\}$ maintains a positive distance from $\partial \mathcal{G}$. Indeed, assume that there exists a sequence $\{w_n\} \subset \partial \mathcal{G}$ satisfying $v_n - w_n \rightarrow 0$ in $H_0^1(\Omega)$ (up to a subsequence). The boundedness of $\{v_n\}$ ensures that $\{w_n\}$ is likewise bounded in $H_0^1(\Omega)$. Consequently,
\[
\inf_{v \in \mathcal{G}} I(v) = \nu_c = \lim_{n \rightarrow \infty} I(v_n) = \lim_{n \rightarrow \infty} I(w_n) \geq \inf_{v \in \partial \mathcal{G}} I(v),
\]
which contradicts inequality~\eqref{5af}. By Ekeland variational principle, we can assume that $(\left.I\right|_{S_c^+})^{\prime}\left(v_n\right)=\left(\left.I\right|_{\mathcal{G}}\right)^{\prime}\left(v_n\right) \rightarrow 0$ as $n \rightarrow \infty$. Up to subsequences,
    we assume that
    \begin{align*}
        & v_n \rightharpoonup v_c \quad \text { weakly in } H_0^1(\Omega), \\
        & v_n \rightharpoonup v_c \quad \text { weakly in } L^{2^*}(\Omega), \\
        & v_n \rightarrow v_c \quad \text { strongly in } L^r(\Omega) \text { for } 2<r<2^*, \\
        & v_n \rightarrow v_c \quad \text { almost everywhere in } \Omega.
    \end{align*}

    On the one hand, we can verify that $v_c \in S_c^+$ is a critical point of $I$ constrained on $S_c^+$. Using Lemma \ref{D} (1), we get that $v_c \in \mathcal{G}$ and so $I(v_c)\geq \nu_c$.

    On the other hand, let $w_n=v_n-v_c$. Since $(\left.I\right|_{S_c^+})^{\prime}\left(v_n\right)\to 0 \ as\  n\to \infty$, there exists $\lambda_n$ such that $I'(v_n)-\lambda_nf(v_n^+)f'(v_n^+) \to 0 \ as \ n \to \infty$.  Let $\lambda_c$ be the Lagrange multiplier corresponding to $v_c$. Then for some $\psi \in H_0^1(\Omega)$ with $\int_{\Omega}f(v_c^+)f'(v_c^+)\psi dx \ne 0$, we obtain 
    \begin{align*}
        &\lambda_n=\frac{1}{\int_{\Omega}f(v_n^+)f'(v_n^+)\psi dx} \left (  \left \langle I'(v_n),\psi \right \rangle+o_n(1) \right )\\
        \overset{n\to\infty}{\longrightarrow}&\frac{1}{\int_{\Omega}f(v_c^+)f'(v_c^+)\psi dx} \left (  \left \langle I'(v_c),\psi \right \rangle+o_n(1) \right )=\lambda _c.
    \end{align*}
    Using the Brézis-Lieb Lemma \cite{Brezis1983} and the facts that
    $$
    I'(v_n)-\lambda_nf(v_n^+)f'(v_n^+) \to 0,\  I'(v_c)-\lambda_cf(v_c^+)f'(v_c^+) = 0,
    $$
    we get 
    $$
    \int_{\Omega}|\nabla w_n|^2 dx=\int_{\Omega}|f(w_n)|^{2(2^*)-2}f(w_n)f'(w_n)w_ndx+o_n(1).
    $$
    Assume that $\int_{\Omega}|\nabla w_n|^2 dx\to l\ge 0,\ \int_{\Omega}|f(w_n)|^{2(2^*)-2}f(w_n)f'(w_n)w_ndx \to l\ge 0$. Using Brézis-Lieb Lemma again, we deduce that
    $$
    I(v_n)=I(v_c)+I(w_n)+o_n(1).
    $$
    Using Lemma \ref{C}, we have 
    \begin{align*}
        I(w_n)&=\frac{1}{2}\int_{\Omega}|\nabla w_n|^2 dx-\frac{1}{2\cdot 2^*}\int_{\Omega}|f(w_n)|^{2(2^*)}dx+o_n(1)\\
        &\ge \frac{1}{2}\int_{\Omega}|\nabla w_n|^2 dx-\frac{1}{2^*}\int_{\Omega}|f(w_n)|^{2(2^*)-2}f(w_n)f'(w_n)w_ndx+o_n(1) \\
        &=\left (\frac{1}{2}-\frac{1}{2^*}\right ) l+o_n(1)=\frac{l}{N}+o_n(1),
    \end{align*}
    that implies $I(w_n)\ge o_n(1)$, and thus $\nu_c=\lim_{n \to \infty} I(v_n)\ge I(v_c).$ Hence, $I(v_c)=\nu_c$, which in turn shows that $l = 0$ and therefore $v_n \to v_c$ strongly in $H_0^1(\Omega)$. By Lagrange multiplier principle, for some $\lambda_c\in \mathbb{R}$, $v_c$ satisfies
    \begin{align}\label{5ag}
        -\Delta v_c-|f(v_c^+)|^{2(2^*)-2}f(v_c^+)f'(v_c^+) =\lambda_c f(v_c^+)f'(v_c^+).
    \end{align}
    From the definition of $f$, we know that $v > 0$ is equivalent to $f(v) > 0$. Multiplying \eqref{5ag} by $v_c^-$ and integrating over $\Omega$, we find $\int_{\Omega}|\nabla v_c^-|^2 dx=0$. Hence $v_c^-= 0$, $\int_{\Omega}|f(v_c)|^2 dx=\int_{\Omega}|f(v_c^+)|^2 dx=c$ and $v_c$ is a solution of \eqref{5ad}. By strong maximum principle $v_c$ is positive. 
    
    Next, we show that $v_c$ is a normalized ground state, i.e.,
    $$
    \tilde{I} (v_c)=\inf \left \{ \tilde{I} (v):v\in S_c,(\tilde{I} |_{S_c})^{\prime}\left(v\right)=0\right \}, 
    $$
    where 
    $$
    \tilde{I}(v)=\frac{1}{2}\int_{\Omega}|\nabla v|^2 d x-\frac{1}{2\cdot 2^*}\int_{\Omega}|f(v)|^{2(2^*)}dx.
    $$
    Since $v_c$ is positive, $I(v_c)=\tilde{I}(v_c) $. Let $v\in S_c$ satisfy $(\left.I\right|_{S_c^+})^{\prime}\left(v\right)= 0$. Then $|v|\in S_c^+$. From the identity $\int_{\Omega}|\nabla v|^2 dx=\int_{\Omega}|\nabla |v||^2 dx$, it can be verified that $|v|\in \mathcal{G}$ and 
    $$
    \tilde{I}(v_c)=I(v_c)\le I(|v|) =\tilde{I}(|v|)= \tilde{I}(v).
    $$
    From the arbitrariness of $v$, we know $v_c$ is a normalized ground state to \eqref{5ad}. Hence, $u_c=f(v_c)$ is a normalized ground state solution to \eqref{5aa}.

    Finally, we prove that $\lambda_c<\lambda_1(\Omega)$. 
    Let $e_1$ be the corresponding positive
    unit eigenfunction of $\lambda_1(\Omega)$.  Multiplying equation \eqref{5ad} for $v_c$ by $e_1$ and integrating over $\Omega$, we obtain 
    \begin{align*}
        &\lambda_c\int_{\Omega}f(v_c)f'(v_c)e_1 dx+\int_{\Omega}|f(v_c)|^{2(2^*)-2}f(v_c)f'(v_c)e_1dx\\
        =&\int_{\Omega}\nabla v_c\nabla e_1 dx=\lambda_1(\Omega)\int_{\Omega}v_ce_1 dx.
    \end{align*}
     Combining Lemma \ref{C}, we conclude $\lambda_c<\lambda_1(\Omega)$. 
\end{proof}

\section{The second positive normalized solution of M-P type}
After finding that $v_c$ is a local minimizer on $S_c^+$, we can construct a mountain pass structure on $S_c^+$:
$$
m(c):=\inf _{\gamma  \in \Gamma} \sup _{t \in[0,1]} I(\gamma (t))>\max \left \{ I(v_c), I(w)\right \}, 
$$
where
\begin{align}
    \Gamma:=\left\{\gamma \in C\left([0,1], S_c^+\right) \text { s.t. } \gamma (0) =u_c, \gamma (1) =w \right\}.\label{5ca}
\end{align}
Lemma \ref{G} ensures the existence of $w$. 
For the normalized mountain pass solution, it seems difficult to show the boundedness of the Palais-Smale sequence for $I(v)$. To overcome this difficulty, we  use the following abstract result \cite{jean1999} to construct a special Palais-Smale sequence for $I(v)$.
\begin{theorem}\label{E}
(Monotonicity trick). Let $J \subset \mathbb{R}^{+}$be an interval. We consider a family $\left(I_\theta\right)_{\theta \in I}$ of $C^1$ functionals on $H_0^1(\Omega)$ of the form
$$
I_\theta(v)=A(v)-\theta B(v), \ \theta \in J,
$$
where $B(u) \geq 0, \forall u \in W$ and such that either $A(u) \rightarrow \infty$ or $B(u) \rightarrow \infty$ as $\|u\| \rightarrow$ $\infty$. We assume there are two points $\left(v_1, v_2\right)$ in $S_c^{+}$(independent of $\theta$ ) such that setting
$$
\Gamma=\left\{\gamma \in C\left([0,1], S_c^{+}\right), \gamma(0)=v_1, \gamma(1)=v_2\right\}
$$
there holds, $\forall \tau \in J$,
$$
m_\theta:=\inf _{\gamma \in \Gamma} \sup _{t \in[0,1]} I_\theta(\gamma(t))>\max \left\{I_\theta\left(v_1\right), I_\theta\left(v_2\right)\right\}.
$$
Then, for almost every $\theta \in J$, there is a sequence $\left\{v_n\right\} \subset S_c^{+}$such that\\
(i) $\left\{v_n\right\}$ is bounded in $H_0^1(\Omega)$, (ii) $E_\theta\left(v_n\right) \rightarrow m_\theta$, (iii) $\left.E_\theta^{\prime}\right|_{S_c^{+}}\left(v_n\right) \rightarrow 0$ in $H_0^1(\Omega)^{\prime}$.
\end{theorem}
To apply the above theorem, we introduce the family of functionals 
\begin{align*}
    I_\theta(u)= \frac{1}{2}\int_{\Omega}|\nabla v|^2 dx-\frac{\theta}{2\cdot2^*}\int_{\Omega}|f(v^+)|^{2(2^*)}dx,
\end{align*}
where $\theta \in [1/2,1]$. Similar to the case that $\theta=1$, we define
$$
\mathcal{G}_\theta:=\left\{v \in S_c^+: \int_{\Omega}\left|\nabla v\right|^2 dx
>\frac{\theta}{2}\int_{\Omega}|f(v)|^{2(2^*)}dx\right\}.\\
$$
Note that $\mathcal{G} \subset \mathcal{G}_\theta$ for $\theta<1$. Hence, $\mathcal{G}_\theta$ is not empty when $\mathcal{G}\ne \emptyset$. We further set $\nu_{c,\theta}:=\underset{\mathcal{G}_\theta}{\inf} I_\theta.$

The following Lemma ensures that $I_\theta(v)$ have M-P geometry.

\begin{lemma}[Uniform M-P geometry.]\label{G}
    Under the assumptions of Lemma \ref{F}, there exist $\epsilon \in (0, 1/2)$ and $\delta > 0$ independent of $\theta$ such that 
    \begin{align}
        I_\theta\left(v_c\right)+\delta<\inf _{\partial \mathcal{G}} I_\theta,\  \forall \theta \in(1-\epsilon, 1],\label{5ah}
    \end{align}
    and there exists $w \in S_c^{+} \backslash \mathcal{G}$ such that
    $$
    m_\theta:=\inf _{\gamma \in \Gamma} \sup _{t \in[0,1]} I_\theta(\gamma(t))>I_\theta\left(v_c\right)+\delta=\max \left\{I_\theta\left(v_c\right), I_\theta(w)\right\}+\delta,
    $$
    where
    $$
    \Gamma:=\left\{\gamma \in C\left([0,1], S_c^{+}\right): \gamma(0)=v_c, \gamma(1)=w\right\}
    $$
    is independent of $\theta$.
\end{lemma}

\begin{proof}
    Note that $\lim _{\theta \rightarrow 1^{-}} I_\theta\left(v_c\right)=I\left(v_c\right)=\nu_c$. Similar to the argument in \cite[Lemma 4.1]{song2024}, we can prove that $\underset{\theta\to 1^-}{\lim}I_\theta=\underset{\partial \mathcal{G}}{\inf}I$. Then combining Lemma \ref{D} (2), we have
    $$
    \lim _{\theta \rightarrow 1^{-}} I_\theta\left(u_c\right)=\nu_c<\inf _{\partial\mathcal{G}} I=\lim _{\theta \rightarrow 1^{-}} \inf _{\partial \mathcal{G}} I_\theta .
    $$
    By choosing $2 \delta=\inf _{\partial \mathcal{G}} I-\nu_c$ and $\epsilon$ small enough, we get \eqref{5ah}.
    Define 
    \begin{align*}
        E_\theta\left(u\right) = \frac{1}{2}\int_{\Omega}|\nabla u|^2 dx+\int_{\Omega}(u)^2|\nabla u|^2 dx-\frac{\theta}{2\cdot2^*}\int_{\Omega }|{u}^+|^{2(2^{*})}dx.      
    \end{align*}
    For $w_1 \in \tilde{S}_c^{+}$, we recall that
    $$
    w_1^t=t^{\frac{N}{2}} w_1(t x) \in \tilde{S}_c^{+}, t \geq 1 .
    $$
    Then, as $t \rightarrow \infty$,
    \begin{align*}
        E_\theta\left(w_1^t\right)
        =& \frac{t^2}{2}\int_{\Omega}|\nabla w_1|^2 dx+t^{N+2}\int_{\Omega}(w_1)^2|\nabla w_1|^2 dx-\frac{t^{(2^*-1)N}\theta }{2\cdot2^*}\int_{\Omega }|w_1^+|^{2(2^{*})}dx\\
        & \rightarrow-\infty \text { uniformly w.r.t. } \theta \in\left[\frac{1}{2}, 1\right] .
    \end{align*}

    Take $w=f^{-1}(w_1^t)$ with $t$ large enough such that $I_\theta(w)=E_\theta\left(w_1^t\right)<I_\theta\left(v_c\right)$ and $w \notin \mathcal{G}$. Then for any $\gamma \in \Gamma$, there exists $t^* \in(0,1)$ such that $\gamma\left(t^*\right) \in \partial \mathcal{G}$. Hence,
    $$
    \inf _{\gamma \in \Gamma} \sup _{t \in[0,1]} I_\theta(\gamma(t)) \geq \inf _{v \in \partial \mathcal{G}} I_\theta(v) .
    $$
    By \eqref{5ah} we complete the proof.    
\end{proof}

\begin{proposition}\label{J}
    Under the assumptions of Lemma \ref{G}, 
    for almost every $\theta \in[1-\epsilon, 1]$ where $\epsilon$ is given by Lemma \ref{G} , there exists $a$ critical point $v_\theta$ of $I_\theta$ constrained on $S_c^{+}$at the level $m_\theta$, which solves
    \begin{equation}
        \left\{\begin{aligned}
            &-\Delta v_\theta -|f(v_\theta)|^{2(2^*)-2}f(v_\theta)f'(v_\theta) =\lambda_\theta f(v_\theta)f'(v_\theta),\text { in } \Omega,\\
            &v_\theta=0 \text { on } \partial \Omega,\ v_\theta>0.\\
        \end{aligned}
        \right.\label{5ai}
    \end{equation}
    Furthermore, $v_\theta \in \mathcal{G}_\theta$.
\end{proposition}

\begin{proof}
    We apply Theorem \ref{E} with
    \begin{align*}
        A(v)=\frac{1}{2} \int_{\Omega}|\nabla v|^2 d x, \ \ B(v)=\int_{\Omega}|f(v^+)|^{2(2^*)}dx.
    \end{align*}
    $B(v)>0$ is obvious. Using Lemma \ref{C}, we have 
    \begin{align*}
        \|v\|^2 &= \int_{\Omega} |\nabla v|^2  dx + \int_{\{x \mid |v(x)| \leq 1\}} v^2  dx + \int_{\{x \mid |v(x)| > 1\}} v^2  dx \\
        &\leq \int_{\Omega} |\nabla v|^2  dx + C_1 \int_{\{x \mid |v(x)| \leq 1\}} f^2(v)  dx + \int_{\{x \mid |v(x)| > 1\}} |v|^{2^*}  dx \\
        &\leq \int_{\Omega} |\nabla v|^2  dx + C_1 \int_{\Omega} f^2(v)  dx + C_2 \left( \int_{\Omega} |\nabla v|^2  dx \right)^{2^*/2} \\
        &\leq C_3 \left(C_4+ A(v) + A(v)^{2^*/2} \right),
    \end{align*}
     which implies that $A(v)\to \infty$ as $\left \| v \right \|\to \infty$.    
    Together with Lemma \ref{G}, we have for almost every $\theta \in[1-\epsilon, 1]$, there exists a bounded (PS) sequence $\left\{v_n\right\} \subset S_c^{+}$ satisfying $I_\theta\left(v_n\right) \rightarrow m_\theta$ and $(\left.I_\theta\right|_{S_c^{+}})^{\prime}\left(v_n\right) \rightarrow 0$ as $n \rightarrow \infty$. Up to a subsequence, we have
    \begin{align*}
        &v_n\rightharpoonup v_\theta\text { weakly in } H_0^1(\Omega),\\
        &v_n\rightharpoonup v_\theta\text { weakly in } L^{2^*}(\Omega),\\
        &v_n\to  v_\theta\text { strongly in } L^r(\Omega),\ \text{ for } 2<r<2^*, \\
        &v_n\to  v_\theta\text { almost everywhere in }  \Omega,
    \end{align*}
    It can be verified that $v_\theta \in S_c^{+}$ is a critical point of $E_\theta$ constrained on $S_c^{+}$. Define $w_n=$ $v_n-v_\theta$. Since $\left(\left.I_\theta\right|_{S_c^{+}}\right)^{\prime}\left(v_n\right) \rightarrow 0$, there exists $\lambda_n$ such that $I_\theta^{\prime}\left(v_n\right)-\lambda_n v_n^{+} \rightarrow 0$. Let $\lambda_\theta$ be the Lagrange multiplier correspond to $v_\theta$. Similar to the proof of Theorem \ref{C}, we have 
    $$
    \int_{\Omega}|\nabla w_n|^2 dx=\theta\int_{\Omega}|f(w_n)|^{2^*-2}f(w_n)f'(w_n)w_ndx+o_n(1).
    $$
    Assume that $\int_{\Omega}\left|\nabla w_n\right|^2 d x \rightarrow \theta l \geq 0, \int_{\Omega}|f(w_n)|^{2^*-2}f(w_n)f'(w_n)w_n \rightarrow l \geq 0$. Combining with the definition of $\mathcal{S}$ and Lemma \ref{C}, we deduce that
    $$
    \int_{\Omega}\left|\nabla w_n\right|^2 d x \geq \mathcal{S}\left(\int_{\Omega} |w_n|^{2^*} d x\right)^{\frac{2}{2^*}}\geq \mathcal{S}\left(\int_{\Omega} |f(w_n)|^{2^*} d x \right)^{\frac{2}{2^*}}\geq \mathcal{S}l^\frac{2}{2^*}.
    $$
    This implies $\theta l \geq \mathcal{S}l^\frac{2}{2^*}$. Suppose that $l>0$, then $ l \geq \mathcal{S}^{\frac{N}{2} }\theta^{-\frac{N}{2}}$, implying that
    $$
    I_\theta\left(w_n\right) \geq \frac{1}{N}\mathcal{S}^{\frac{N}{2} }\theta^{\frac{2-N}{2}}+o_n(1).
    $$
    Further, using the Brézis-Lieb Lemma one gets
    \begin{align}\label{5ba}
    I_\theta\left(v_n\right)=I_\theta\left(v_\theta\right)+I_\theta\left(w_n\right)+o_n(1) \geq \nu_{c, \theta}+\frac{1}{N}\mathcal{S}^{\frac{N}{2}  }\theta^{\frac{2-N}{2}}+o_n(1) .
    \end{align}
    
    On the other hand, similar to \cite[Lemma 4.1]{song2024},  we can prove $\underset{\theta \to 1^-}{\lim} \nu_{c,\theta}=\nu_c$. Using Lemma \ref{F} which will be proved later, we have
    $$
    m_\theta=m_1+o_\theta(1)<I\left(v_c\right)+\frac{1}{2N}\mathcal{S}^{\frac{N}{2} }+o_\theta(1)=\nu_{c, \theta}+\frac{1}{2N}\mathcal{S}^{\frac{N}{2} }\theta^{\frac{2-N}{2}}+o_\theta(1) .
    $$
    By choosing $\epsilon$ small enough such that $\theta$ close to $1^{-}$ for all $\theta \in(1-\epsilon, 1)$, we obtain
    \begin{align}\label{5bb}
        m_\theta+\xi<\nu_{c, \theta}+\frac{1}{2N} \mathcal{S}^{\frac{N}{2} }\theta^{\frac{2-N}{2}},
    \end{align}
    where some $\xi>0$ independent of $\theta \in(1-\epsilon, 1)$. Combining \eqref{5ba}, \eqref{5bb} and that $\lim _{n \rightarrow \infty} I_\theta\left(v_n\right)=m_\theta$, we arrive at a contradictory chain of inequalities:
    $$
    \nu_{c, \theta}+\frac{1}{N} \mathcal{S}^{\frac{N}{2} }\theta^{\frac{2-N}{2}} <m_\theta+\xi<\nu_{c, \theta}+\frac{1}{2N}\mathcal{S}^{\frac{N}{2} }\theta^{\frac{2-N}{2}}.
    $$
    Thus we prove that $l=0$. This shows $v_n \rightarrow v_\theta$ strongly in $H_0^1(\Omega)$ and $v_\theta$ is a critical point of $I_\theta$ constrained on $S_c^{+}$at the level $m_\theta$.
    By Lagrange multiplier principle, $v_\theta$ satisfies
    $$
    -\Delta v_\theta -|f(v_\theta^+)|^{2(2^*)-2}f(v_\theta^+)f'(v_\theta^+) =\lambda f(v_\theta^+)f'(v_\theta^+),\ \text { in } \Omega.
    $$
    Multiplying $v_\theta^{-}$ for the equation of $v_\theta$ and integrating over $\Omega$, we get
    $$
    \int_{\Omega}\left|\nabla v_\theta^{-}\right|^2 d x=0
    $$
    By strong maximum principle, $v_\theta>0$ and thus solves \eqref{5ai}. Using the Pohozaev identity and that $\Omega$ is star-shaped with respect to 0 , we know $v_\theta \in \mathcal{G}_\theta$ and complete the proof.
\end{proof}

\begin{lemma}\label{F}
    Let $N \geq 3$, under the assumptions of Theorem \ref{A} and Theorem \ref{B}. Then 
    $$
    m(c) \leq \nu_c + \frac{1}{2N} S^{\frac{N}{2}}+o(1).
    $$
\end{lemma}

\begin{proof}
    Let  $\xi \in C_{0}^{\infty}(\Omega)$  be the radial function, such that  $\xi(x) \equiv 1$  for  $0 \leq|x| \leq R,\ 0 \leq   \xi(x) \leq 1$  for $R \leq|x| \leq 2 R, \xi(x) \equiv 0$ for $|x| \geq 2 R$, where $B_{2 R} \subset \Omega$. Take $v_{\epsilon}=\xi w_{\epsilon}$  where
    $$
    w_{\epsilon}=\frac{\left ( N(N-2)\epsilon \right )^{\frac{N-2}{8} } }{(\epsilon+|x|^{2})^{\frac{N-2}{4} } }.
$$
Note that the function $u_\epsilon = w_\epsilon^2$ solves the equation $\Delta u_\epsilon +u_\epsilon ^{\frac{N+2}{N-2}}=0$. As $\epsilon\to 0$ the function $v_\epsilon$ satisfies the following estimates:
\begin{align}
    &4\int_{\Omega}\left|v_{\epsilon}\right|^{2}\left|\nabla v_{\epsilon}\right|^{2}dx=\int_{\Omega}\left|\nabla v_{\epsilon}^{2}\right|^{2} d x=S^{\frac{N}{2}}+O(\epsilon^{\frac{N-2}{2}}),\label{5aj} \\
    &\int_{\Omega} v_{\epsilon}^{2\left(2^{*}\right)} d x=S^{\frac{N}{2}}+O\left(\epsilon^{\frac{N}{2}}\right),\label{5ak} \\
    &\int_{\Omega}\left|\nabla v_{\epsilon}\right|^{2} d x=O\left(\epsilon^{\frac{N-2}{4}} |\ln \epsilon|\right), \label{5al}        
\end{align}
\begin{equation}\label{5am}
    \int_{\Omega} v_{\epsilon}^{r} d x \approx
    \left\{\begin{aligned}
        &\epsilon^{\frac{N}{2}-\frac{1}{8} r(N-2)},\ 2^{*}<r<2\left(2^{*}\right),\\
        &\epsilon^{\frac{N}{4}}|\ln \epsilon|,\  r=2^{*},\\
        &\epsilon^{\frac{1}{8}r(N-2)},\ r<2^{*}.
    \end{aligned}
    \right.
\end{equation}

    Let $w_{\epsilon, s}=u_c+s v_\epsilon,\ s \geq 0$ where $u_c$ is given by Theorem \ref{A}. Then $w_{\epsilon, s}>0$ in $\Omega$. Define $W_{\epsilon, s}=\mu ^{\frac{N-2}{4}} w_{\epsilon, s}(\mu x)$ with the scaling parameter $\mu$ chosen as $\mu=\left (\frac{\left\|w_{\epsilon, s}\right\|_{L^2(\Omega)}^2}{c}\right )^\frac{2}{N+2}$, we obtain that $\left\|W_{\epsilon, s}\right\|_{L^2(\Omega)}^2=c$. It is easy to verify that $\left\|W_{\epsilon, s}\right\|_{L^2(\Omega)}^2=c$.

    Let $\overline{W}_k=k^{\frac{N}{2}} W_{\epsilon, \widehat{s}}(k x),\ k \geq 1$ where $\widehat{s}$ will be determined below. Set
    $$
    \phi(k):=E(\overline{W}_k)=\frac{k^2}{2}\int_{\Omega}|\nabla W_{\epsilon, \hat{s} }|^2 d x+k^{N+2}\int_{\Omega}W_{\epsilon, \hat{s}}^2|\nabla W_{\epsilon, \hat{s}}|^2 d x-\frac{k^{(2^*-1)N}}{2\cdot 2^*}\int_{\Omega}|W_{\epsilon, \hat{s}}|^{2(2^*)} d x,\ \ k\ge 1. 
    $$
    Then
    \begin{align*}
        \phi'(k)=k\int_{\Omega}|\nabla W_{\epsilon, \hat{s} }|^2 d x+(N+2)k^{N+1}\int_{\Omega}W_{\epsilon, \hat{s} }^2|\nabla W_{\epsilon, \hat{s} }|^2 d x-\frac{(2^*-1)Nk^{(2^*-1)N-1}}{2\cdot 2^*} \int_{\Omega}|W_{\epsilon, \hat{s} }|^{2(2^*)} dx.
    \end{align*}
    By \eqref{5aj}-\eqref{5al}, we have 
    \begin{align*}
        &\int_{\Omega}|\nabla W_{\epsilon, \hat{s} }|^2 d x=\int_{\Omega}|\nabla u_c|^2 d x+o_\epsilon (1),\\
        &\int_{\Omega}|W_{\epsilon, \hat{s}}|^{2(2^*)}dx=\int_{\Omega}|u_c|^{2(2^*)}dx+\hat{s}^{2(2^*)}\mathcal{S}^{\frac{N}{2} }+o_\epsilon (1),\\
        & \int_{\Omega}|W_{\epsilon, \hat{s}}|^2 |\nabla W_{\epsilon, \hat{s} }|^2dx=\int_{\Omega}|u_c|^2 |\nabla u_c|^2dx+\frac{\hat{s}^4}{4}\mathcal{S}^{\frac{N}{2} } +o_\epsilon (1).
    \end{align*}
    We may choose $k$ large such that $\varphi'(k) < 0$ for all $k > 1$. Hence, $E(\overline{W}_{k}) \le E(W_{\varepsilon, \hat{s}})$. Furthermore, it is not difficult to see that $\varphi(k) \to -\infty$ as $k \to \infty$. Define $\gamma(t) := W_{\varepsilon, 2t\hat{s}}$ for $t \in [0, 1/2]$ and $\gamma(t) := \overline{W}_{2(t-1/2)k_0+1}$ for $t \in [1/2,1]$, 
    where $k_0$ is large enough such that $E\left(\overline{W}_{k_0+1}\right) < E(u_c)$ and $\overline{W}_{k_0+1} \notin \mathcal{G}$. Then $\gamma \in \Gamma$, where $\Gamma$ is defined by \eqref{5ca}. 

    We claim that 
    \begin{align}\label{5an}
    E\left(W_{\epsilon, s}\right)<E\left(u_c\right)+\frac{1}{2N}\mathcal{S}^{\frac{N}{2}}+o(1).    
    \end{align}
     If this claim holds, then $\sup _{t \in[0,1]} E(\gamma)<E \left (u_c\right)+\frac{1}{2N}\mathcal{S}^{\frac{N}{2}}$, implying that $m(c)<\nu_c+\frac{1}{2N}\mathcal{S}^{\frac{N}{2}}.$ In order to conclude, it is sufficient to prove \eqref{5an}. 
     
    Note that
    \begin{align*}
        E\left(W_{\epsilon, s}\right)= &E\left(w_{\epsilon, s}\right)+\frac{1}{2} \left(1-\mu^\frac{2-N}{2}\right)\int_{\Omega}\left|\nabla w_{\epsilon, s}\right|^2 \nonumber \\
        &= E(u_c) + \frac{1}{2} \left( \mu ^{\frac{2-N}{2}} - 1 \right) \int_{\Omega} |\nabla u_c|^2 dx + s\mu^{\frac{2-N}{2}} \int_{\Omega} \nabla u_c \nabla v_\epsilon dx \\
        &+ \frac{s^2}{2} \mu^{\frac{2-N}{2}} \int_{\Omega } |\nabla v_\epsilon|^2 dx + \frac{s^4}{4} \int_{\Omega} |\nabla v_\epsilon|^2 dx + s^2 \int_{\Omega} |\nabla (u_c v_\epsilon)|^2 dx \\
        &+ \frac{s^2}{2} \int_{\Omega} \nabla u_c^2 \nabla v_\epsilon^2 dx + s \int_{\Omega} \nabla u_c^2 \nabla (u_c v_\epsilon) dx+ s^3 \int_{\Omega} \nabla v_\epsilon^2 \nabla (u_c v_\epsilon) dx \\
        &+\frac{1}{2 \cdot 2^*}  \int_{\Omega} \left |u_c\right | ^{2 ( 2^*)} dx- \frac{1}{2 \cdot 2^*}  \int_{\Omega} \left |u_c+sv_\epsilon\right | ^{2 ( 2^*)} dx.
    \end{align*}
    Since $u_c$ satisfies \eqref{5aa}, one gets 
    \begin{align}\label{5ao}
        \int_{\Omega}\nabla u_c \nabla v_\epsilon dx+\int_{\Omega}\nabla u_c^2 \nabla (u_cv_\epsilon)dx=\int_{\Omega}|u_c|^{2(2^*)-2}u_cv_\epsilon dx+\lambda_c \int_{\Omega}u_cv_\epsilon dx.
    \end{align}
    Noticing that 
    $$
    \mu^\frac{N+2}{2}=\frac{\left \|u_c+sv_\epsilon\right \|_2^2 }{c}=1+\frac{2s}{c} \int_{\Omega} u_c v_\epsilon d x+\frac{s^2}{c} \int_{\Omega}\left|v_\epsilon\right|^2 d x.
    $$
    It follows from \eqref{5am} that
    \begin{align}\label{5ap}
        \mu^\frac{N-2}{2}-1=\mu^{\frac{N+2}{2}\frac{2-N}{N+2}}-1=\frac{s}{c}\frac{2(2-N)}{N+2}\int_{\Omega}u_cv_\epsilon dx-s^2O(\epsilon^\frac{N-2}{4}).
    \end{align}
    Using $(a+b)^r\ge a^r+b^r+rab^{r-1}$ when $r \ge 2$, we have 
    \begin{align}\label{5aq}
        &\frac{1}{2\cdot 2^*} \int_{\Omega}\left (|u_c|^{2(2^*)}-|u_c+s v_\epsilon|^{2(2^*)}\right ) dx \nonumber \\
        \le &-\frac{s^{2(2^*)}}{2\cdot 2^*}\int_{\Omega}|v_\epsilon|^{2(2^*)}dx-s^{2(2^*)-1}\int_{\Omega}u_cv_\epsilon^{2(2^*)-1}dx-\int_{\Omega}u_c^{2(2^*)-1}v_\epsilon dx.
    \end{align}
    From \eqref{5ao}-\eqref{5aq}, we get that
    \begin{align*}
        E(W_{\varepsilon,t}) \leq &\ E(u_c)+ s^2 \int_{\Omega} u_c^2 |\nabla v_\epsilon|^2 dx+ s^2 \int_{\Omega} |\nabla u_c|^2 v_\epsilon^2 dx  \\
        & + s^2 \int_{\Omega} \nabla u_c^2 \nabla v_\epsilon^2 dx + 2s^3 \int_{\Omega} v_\epsilon^2 \nabla u_c \nabla v_\epsilon dx+ 2s^3 \int_{\Omega} |\nabla v_\epsilon|^2 u_c v_\epsilon dx \\
        & + \lambda_c s \int_{\Omega} u_c v_\epsilon dx + \frac{1}{2} \left( \mu^{\frac{2-N}{2}} - 1 \right) \int_{\Omega} |\nabla u_c|^2 dx\\
        & + \left( \mu^{\frac{2-N}{2}} - 1 \right) s \int_{\Omega} \nabla u_c \nabla v_\epsilon dx + \frac{s^2}{2} \mu^{\frac{2-N}{2}} \int_{\Omega} |\nabla v_\epsilon|^2 dx\\
        & + \frac{s^4}{4} \int_{\Omega} |\nabla v_\epsilon^2|^2 dx - \frac{s^{2\cdot2^*}}{2 \cdot 2^*} \int_{\Omega} v_\epsilon^{2\cdot2^*} dx- s^{2\cdot2^*-1} \int_{\Omega} u_c v_\epsilon^{2\cdot2^*-1} dx        
\end{align*} 
Similar to \cite[Lemma 3.13]{chang2025}, as $\epsilon\to 0$, we have the following estimates: 
\begin{align}
    &\int_{\Omega} u^2_c |\nabla v_\epsilon|^2  dx = O ( \epsilon^{\frac{N-2}{4}} |\ln \epsilon|), \label{5ar}\\
    &\int_{\Omega} |\nabla u_c|^2 v_\epsilon^2  dx = O ( \epsilon^{\frac{N-2}{4}}),\label{5as}\\
    &\int_{\Omega} \nabla u^2_c \cdot \nabla v_\epsilon^2  dx = O ( \epsilon^{\frac{N-2}{4}} \sqrt{|\ln \epsilon|}),\label{5at} \\
    &\int_{\mathbb{R}^{N}} v_{\epsilon}^{2} \nabla u_c \nabla v_\epsilon :=\zeta_\epsilon (N)= 
    \begin{cases} 
    O(\epsilon^{\frac{N+2}{8}}) & \text{if } N > 4, \\
    O(\epsilon^{\frac{3}{4}} | \ln \epsilon |) & \text{if } N = 4, \\
    O(\epsilon^{\frac{3}{8}}) & \text{if } N = 3,
    \end{cases} \label{5au}\\
    &\int_{\mathbb{R}^{N}} |\nabla v_{\epsilon}|^{2} u_c v_{\epsilon}  dx= O( \epsilon^{\frac{N-2}{8}}), \label{5av} \\
    &\int_{\Omega} u_c v^{2 \cdot 2^* - 1}_{\epsilon}  dx = O\left( \epsilon^{\frac{N-2}{8}} \right). \label{5aw}
\end{align} 
In view of \eqref{5aj}-\eqref{5am} and \eqref{5ar}-\eqref{5aw}, we get
\begin{align*}
    E(W_{\varepsilon,t})\leq &E(u_c) + \frac{s^4}{4} \mathcal{S}^{\frac{N}{2}} -\frac{s^{2\cdot2^*}}{2 \cdot 2^*} \mathcal{S}^{\frac{N}{2}}+o_\epsilon(1)\leq E\left(u_c\right)+\frac{1}{2N}\mathcal{S}^{\frac{N}{2}}+o_\epsilon(1).
\end{align*}    
    We complete the proof.
\end{proof}

Now we are prepared to prove Theorem \ref{B}.

\begin{proof}[proof of Theorem \ref{B}]
    Firstly, we construct a Palais-Smale sequence at the energy level $m(c)$. By Lemma \ref{G}, we can take $\theta_n \rightarrow 1^{-}$and $v_n=v_{\theta_n}$ to solve \eqref{5ai}. Notice that $\lim_{\theta \to 1^{-}} m_{\theta} = m_{1}$. Taking $n$ large such that $\theta_n$ close to $1^{-}$ enough, using the fact that $v_{\theta_n} \in \mathcal{G}_{\theta_n}$, we obtain
    \begin{align*}
        2 m_1 \geq m_{\theta_n}=I_{\theta_n}\left(v_n\right)  >\frac{1}{N}\int_{\Omega}|\nabla v_n|^2 dx,
    \end{align*}
    yielding the $H_0^1(\Omega)$-boundedness of $\left\{v_n\right\}$. Then, we have
    $$
    (\left.I\right|_{S_c^{+}})^{\prime}\left(v_n\right)=(\left.I_{\theta_n}\right|_{S_c^{+}})^{\prime}\left(v_n\right)+o_n(1) \rightarrow 0 \text { as } n \rightarrow \infty .
    $$
    $$
    \lim _{n \rightarrow \infty} I\left(v_n\right)=\lim _{n \rightarrow \infty} I_{\theta_n}\left(v_n\right)=\lim _{n \rightarrow \infty} m_{\theta_n}=m_1 .
    $$
    Hence, $v_n$ is exactly the sequence we need. Up to a subsequence, we assume that
    \begin{align*}
        & v_n \rightharpoonup \tilde{v}_c \quad \text { weakly in } H_0^1(\Omega), \\
        & v_n \rightharpoonup \tilde{v}_c \quad \text { weakly in } L^{2^*}(\Omega), \\
        & v_n \rightarrow \tilde{v}_c \quad \text { strongly in } L^r(\Omega) \text { for } 2<r<2^*, \\
        & v_n \rightarrow \tilde{v}_c \quad \text { almost everywhere in } \Omega .
    \end{align*}
    It can be verified that $\tilde{v}_c \in S_c^{+}$ is a positive solution of \eqref{5aa}. Let $w_n=v_n-\tilde{v}_c$. Since $(\left.I\right|_{S_c^{+}})^{\prime}\left(v_n\right) \rightarrow 0$, there exists $\lambda_n$ such that $ I'(v_n)-\lambda_nf(v_n^+)f'(v_n^+) \to 0$. Let $\tilde{\lambda}_c$ be the Lagrange multiplier correspond to $\tilde{v}_c$. Similar to the proof to Theorem \ref{A}, we have
    $$
    \int_{\Omega}|\nabla w_n|^2 dx=\int_{\Omega}|f(w_n)|^{2(2^*)-2}f(w_n)f'(w_n)w_ndx+o_n(1).
    $$
    Hence, we assume that $\int_{\Omega}\left|\nabla w_n\right|^2 d x \rightarrow l \geq 0,\ \int_{\Omega}|f(w_n)|^{2(2^*)-2}f(w_n)f'(w_n)w_n\to l\ge 0$. 
    From the definition of $\mathcal{S}$ and Lemma \ref{C}, we deduce that
    $$
    \int_{\Omega}\left|\nabla w_n\right|^2 d x \geq \mathcal{S}\left(\int_{\Omega} |w_n|^{2^*} d x\right)^{\frac{2}{2^*}}\geq \mathcal{S}\left(\int_{\Omega} |f(w_n)|^{2^*} d x \right)^{\frac{2}{2^*}}\geq \mathcal{S}l^\frac{2}{2^*},
    $$
    implying $ l \geq \mathcal{S}l^\frac{2}{2^*}$. Assume, for the sake of contradiction, that $l>0$, then $ l \geq \mathcal{S}^{\frac{N}{2} }$, implying that
    $$
    I\left(w_n\right) \geq \frac{1}{N}\mathcal{S}^{\frac{N}{2} }+o_n(1).
    $$
    Further, using the Brézis-Lieb Lemma, we get
    $$
    I\left(v_n\right)=I\left(\tilde{v}_c\right)+I\left(w_n\right)+o_n(1) \geq \nu_c+\frac{1}{N}\mathcal{S}^{\frac{N}{2} }+o_n(1) .
    $$
    Then,
    $$
    \nu_c+\frac{1}{N}\mathcal{S}^{\frac{N}{2} } \leq m(c)<\nu_c+\frac{1}{2N}\mathcal{S}^{\frac{N}{2}}+o(1).
    $$
    leads to a contradiction. Therefore, $l=0$ and $v_n \rightarrow \tilde{v}_c$ strongly in $H_0^1(\Omega)$. Similar to Proposition \ref{J}, we get $\tilde{v}_c>0$ solves \eqref{5ad}. Hence, $\tilde{u}_c=f(\tilde{v}_c)$ is a normalized mountain pass solutions. 
    
\end{proof}

\bibliographystyle{unsrt}
\bibliography{ref}

\begin{thebibliography}{10}

\bibitem{wzq2004}
J.Q. Liu, Y.Q. Wang, and Z.Q. Wang.
\newblock Solutions for quasilinear {S}chr\"odinger equations via the {N}ehari method.
\newblock {\em Comm. Partial Differential Equations}, 29(5-6):879--901, 2004.

\bibitem{jean2010}
M.~Colin, L.~Jeanjean, and M.~Squassina.
\newblock Stability and instability results for standing waves of quasilinear {S}chr\"odinger equations.
\newblock {\em Nonlinearity}, 23(6):1353--1385, 2010.

\bibitem{wzq2002}
M.~Poppenberg, K.~Schmitt, and Z.Q. Wang.
\newblock On the existence of soliton solutions to quasilinear {S}chr\"odinger equations.
\newblock {\em Calc. Var. Partial Differential Equations}, 14(3):329--344, 2002.

\bibitem{wzq2003}
J.Q. Liu and Z.Q. Wang.
\newblock Soliton solutions for quasilinear {S}chr\"odinger equations. {I}.
\newblock {\em Proc. Amer. Math. Soc.}, 131(2):441--448, 2003.

\bibitem{Ruiz2010}
D.~Ruiz and G.~Siciliano.
\newblock Existence of ground states for a modified nonlinear {S}chr\"odinger equation.
\newblock {\em Nonlinearity}, 23(5):1221--1233, 2010.

\bibitem{jean2004}
M.~Colin and L.~Jeanjean.
\newblock Solutions for a quasilinear {S}chr\"odinger equation: a dual approach.
\newblock {\em Nonlinear Anal.}, 56(2):213--226, 2004.

\bibitem{wzq2003-2}
J.Q. Liu, Y.Q. Wang, and Z.Q. Wang.
\newblock Soliton solutions for quasilinear {S}chr\"odinger equations. {II}.
\newblock {\em J. Differential Equations}, 187(2):473--493, 2003.

\bibitem{wzq2014}
J.Q. Liu, X.Q. Liu, and Z.Q. Wang.
\newblock Multiple sign-changing solutions for quasilinear elliptic equations via perturbation method.
\newblock {\em Comm. Partial Differential Equations}, 39(12):2216--2239, 2014.

\bibitem{wzq2013}
X.Q. Liu, J.Q. Liu, and Z.Q. Wang.
\newblock Quasilinear elliptic equations via perturbation method.
\newblock {\em Proc. Amer. Math. Soc.}, 141(1):253--263, 2013.

\bibitem{wzq2023}
L.~Zhang, J.Q. Chen, and Z.Q. Wang.
\newblock Ground states for a quasilinear {S}chr\"odinger equation: mass critical and supercritical cases.
\newblock {\em Appl. Math. Lett.}, 145:Paper No. 108763, 7, 2023.

\bibitem{Zou2023}
H.W. Li and W.M. Zou.
\newblock Quasilinear {S}chr\"odinger equations: ground state and infinitely many normalized solutions.
\newblock {\em Pacific J. Math.}, 322(1):99--138, 2023.

\bibitem{Yang2025}
X.Y. Yang and F.K. Zhao.
\newblock Infinitely many normalized solutions for a quasilinear {S}chr\"odinger equation.
\newblock {\em J. Geom. Anal.}, 35(2):Paper No. 52, 32, 2025.

\bibitem{zxx2024}
T.~Deng, M.~Squassina, J.J. Zhang, and X.X. Zhong.
\newblock Normalized solutions of quasilinear {S}chr\"odinger equations with a general nonlinearity.
\newblock {\em Asymptot. Anal.}, 140(1-2):5--24, 2024.

\bibitem{chang2025}
Y.X. Li, M.J. Yang, and Chang X.J.
\newblock Normalized solutions for a sobolev critical quasilinear schr\"odinger equation.
\newblock {\em arXiv:2506.10870}, 2025.

\bibitem{Gyx2024}
F.S. Gao and Y.X. Guo.
\newblock Existence of normalized solutions for mass super-critical quasilinear {S}chr\"odinger equation with potentials.
\newblock {\em J. Geom. Anal.}, 34(11):Paper No. 329, 39, 2024.

\bibitem{Gyx2025}
F.S. Gao and Y.X. Guo.
\newblock Normalized solution for a quasilinear {S}chr\"odinger equation with potentials and general nonlinearities.
\newblock {\em Commun. Pure Appl. Anal.}, 24(4):507--534, 2025.

\bibitem{jean2015}
L.~Jeanjean, T.J. Luo, and Z.Q. Wang.
\newblock Multiple normalized solutions for quasi-linear {S}chr\"odinger equations.
\newblock {\em J. Differential Equations}, 259(8):3894--3928, 2015.

\bibitem{jean1997}
L.~Jeanjean.
\newblock Existence of solutions with prescribed norm for semilinear elliptic equations.
\newblock {\em Nonlinear Anal.}, 28(10):1633--1659, 1997.

\bibitem{Bartch2017}
T.~Bartsch and N.~Soave.
\newblock A natural constraint approach to normalized solutions of nonlinear {S}chr\"odinger equations and systems.
\newblock {\em J. Funct. Anal.}, 272(12):4998--5037, 2017.

\bibitem{Pierotti2017}
D.~Pierotti and G.~Verzini.
\newblock Normalized bound states for the nonlinear {S}chr\"odinger equation in bounded domains.
\newblock {\em Calc. Var. Partial Differential Equations}, 56(5):Paper No. 133, 27, 2017.

\bibitem{Noris2019}
B.~Noris, H.~Tavares, and G.~Verzini.
\newblock Normalized solutions for nonlinear {S}chr\"odinger systems on bounded domains.
\newblock {\em Nonlinearity}, 32(3):1044--1072, 2019.

\bibitem{Noris2014}
B.~Noris, H.~Tavares, and G.~Verzini.
\newblock Existence and orbital stability of the ground states with prescribed mass for the {$L^2$}-critical and supercritical {NLS} on bounded domains.
\newblock {\em Anal. PDE}, 7(8):1807--1838, 2014.

\bibitem{chang20252}
X.J. Chang, M.T. Liu, and D.K. Yan.
\newblock Positive normalized solutions of {S}chr\"{o}dinger equations with sobolev critical growth in bounded domains.
\newblock {\em arXiv:2505.07578}, 2025.

\bibitem{Pellacci2021}
B.~Pellacci, A.~Pistoia, G.~Vaira, and G.~Verzini.
\newblock Normalized concentrating solutions to nonlinear elliptic problems.
\newblock {\em J. Differential Equations}, 275:882--919, 2021.

\bibitem{Bartsch2024}
T.~Bartsch, S.J. Qi, and W.M. Zou.
\newblock Normalized solutions to {S}chr\"odinger equations with potential and inhomogeneous nonlinearities on large smooth domains.
\newblock {\em Math. Ann.}, 390(3):4813--4859, 2024.

\bibitem{Lyy2025}
Y.Y. Liu and L.G. Zhao.
\newblock Normalized solutions for {S}chr\"odinger equations with general nonlinearities on bounded domains.
\newblock {\em J. Geom. Anal.}, 35(2):Paper No. 54, 32, 2025.

\bibitem{song2024}
L.J. Song and W.M. Zou.
\newblock Two positive normalized solutions on star-shaped bounded domains to the {B}r\'ezis-{N}irenberg problem, {I}: Existence.
\newblock {\em arXiv:2404.11204}, 2024.

\bibitem{song2023}
L.J. Song.
\newblock Existence and orbital stability/instability of standing waves with prescribed mass for the {$L^2$}-supercritical {NLS} in bounded domains and exterior domains.
\newblock {\em Calc. Var. Partial Differential Equations}, 62(6):Paper No. 176, 29, 2023.

\bibitem{Do2009}
J.M. do~\'O and U.~Severo.
\newblock Quasilinear {S}chr\"odinger equations involving concave and convex nonlinearities.
\newblock {\em Commun. Pure Appl. Anal.}, 8(2):621--644, 2009.

\bibitem{Brezis1983}
H.~Br\'ezis and E.~Lieb.
\newblock A relation between pointwise convergence of functions and convergence of functionals.
\newblock {\em Proc. Amer. Math. Soc.}, 88(3):486--490, 1983.

\bibitem{jean1999}
L.~Jeanjean.
\newblock On the existence of bounded {P}alais-{S}male sequences and application to a {L}andesman-{L}azer-type problem set on {$\mathbb{R}^N$}.
\newblock {\em Proc. Roy. Soc. Edinburgh Sect. A}, 129(4):787--809, 1999.

\end{thebibliography}

\end{document}